\documentclass[10pt, a4paper]{article}
\usepackage{amsfonts}
\usepackage{mathrsfs}
\usepackage{latexsym}
\usepackage{xy}
\usepackage{amsfonts,amsmath,amssymb,amsthm}
\usepackage{color}

\xyoption{all}

\newcommand{\bcen}{\begin{center}}     \newcommand{\ecen}{\end{center}}
\newcommand{\bay}{\begin{array}}      \newcommand{\eay}{\end{array}}
\newcommand{\beq}{\begin{eqnarray*}}      \newcommand{\eeq}{\end{eqnarray*}}

\def\dz{\delta}

\def\ot{\otimes}

\def\A{\mathcal{A}}
\def\add{\mathrm{add}}

\def\B{\mathcal{B}}

\def\D{\mathcal{D}}

\def\dim{\mathrm{dim}}

\def\dz{\delta}
\def\End{\mathrm{End}}

\def\Ext{\mathrm{Ext}}

\def\gl{\mathrm{gl.dim}}

\def\Hom{\mathrm{Hom}}

\def\mod{\mathrm{mod}}
\def\Mod{\mathrm{Mod}}

\def\ol{\overline}
\def\op{\mathrm{op}}
\def\ot{\otimes}
\def\pd{\mathrm{pd}}
\def\per{\mathrm{per}}

\def\rad{\mathrm{rad}}

\def\RHom{\mathrm{RHom}}

\def\tr{\mathrm{tr}}

\def\ul{\underline}
\def\Z{\mathbb{Z}}

\begin{document}

\newtheorem{theorem}{Theorem}
\newtheorem{proposition}{Proposition}
\newtheorem{lemma}{Lemma}
\newtheorem{corollary}{Corollary}
\newtheorem{remark}{Remark}
\newtheorem{example}{Example}
\newtheorem{definition}{Definition}
\newtheorem*{conjecture}{Conjecture}
\newtheorem{question}{Question}

\title{\bf\large Entropies of Serre functors for higher hereditary algebras \footnote{This work was supported by the National Natural Science Foundation of China (Grant No. 12371043 and 11971460).}}

\author{Yang Han}

\date{\footnotesize KLMM, Academy of Mathematics and Systems Science,
Chinese Academy of Sciences, \\ Beijing 100190, China. \\ School of Mathematical Sciences, University of
Chinese Academy of Sciences, \\ Beijing 100049, China.\\ E-mail: hany@iss.ac.cn}

\maketitle

{\small\it Dedicated to Professor Claus Michael Ringel on the occasion of his 80th birthday}

\bigskip

\begin{abstract}
For a higher hereditary algebra, we calculate its upper (lower) Serre dimension, the entropy and polynomial entropy of Serre functor, and the Hochschild (co)homology entropy of Serre quasi-functor. These invariants are determined by its Calabi-Yau dimension for a higher representation-finite algebra, and by its global dimension and the spectral radius and polynomial growth rate of its Coxeter matrix for a higher representation-infinite algebra. For this, we will prove the Yomdin type inequality on Hochschild homology entropy for a finite dimensional elementary algebra of finite global dimension. Our calculations imply that the Kikuta and Ouchi's question on relations between entropy and Hochschild (co)homology entropy has positive answer, and the Gromov-Yomdin type equalities on entropy and Hochschild (co)homology entropy hold, for the Serre functor on perfect derived category and the Serre quasi-functor on perfect dg module category of an indecomposable elementary higher hereditary algebra.
\end{abstract}

\medskip


\medskip


\tableofcontents

\section{Introduction}

A topological dynamical system $(X,f)$ consists of a topological space $X$ and a continuous function $f:X\to X$. The topological entropy $h_\mathrm{top}(f)$ measures the complexity of $(X,f)$.
As a categorical analog of topological dynamical system, in \cite{DimHaiKatKon14}, Dimitrov, Haiden, Katzarkov and Kontsevich introduced categorical dynamical system $({\cal T},F)$ which consists of a triangulated category ${\cal T}$ and a (triangle) endofunctor $F: {\cal T}\to {\cal T}$ of ${\cal T}$, and (categorical) entropy which measures the complexity of a categorical dynamical system.
Roughly speaking, the entropy is the asymptotically exponential growth rate of the complexity of a categorical dynamical system. In \cite{FanFuOuc21}, Fan, Fu and Ouchi introduced (categorical)
polynomial entropy which is the asymptotically polynomial growth rate of the complexity of a categorical dynamical system. Moreover, in \cite{KikOuc20}, Kikuta and Ouchi introduced Hochschild (co)homology entropy. Hochschild (co)homology entropy is defined not for a categorical dynamical system, but for a ``dg categorical dynamical system''. Meanwhile, Kikuta and Ouchi posed a question (\cite[Question 2.13]{KikOuc20}): When does the Hochschild (co)homology entropy for a ``dg categorical dynamical system'' coincide with the entropy for the corresponding categorical dynamical system? In addition to these types of entropies, in \cite{ElaLun21}, Elagin and Lunts introduced the upper (lower) Serre dimension of a triangulated category with a (split=classical) generator and a Serre functor, which sometimes is the coefficient of the degree one term of the entropy (\cite[Proposition 6.14]{ElaLun21}).

As a generalization of representation-finite hereditary algebras, Iyama and Oppermann introduced higher representation-finite algebras in \cite{IyaOpp11}. As a generalization of representation-infinite hereditary algebras, Herschend, Iyama and Oppermann introduced higher representation-infinite algebras in \cite{HerIyaOpp14}.
Meanwhile, they also introduced higher hereditary algebras which are shown to be either higher representation-finite algebras or higher representation-infinite algebras (\cite[Theorem 3.4]{HerIyaOpp14}). Many classical results in the representation theory of hereditary algebras have higher dimensional analogs for higher hereditary algebras.
Moreover, as a generalization of fractionally Calabi-Yau algebras, Herschend and Iyama introduced twisted fractionally Calabi-Yau algebras in \cite{HerIya11}, which contain higher representation-finite algebras as typical examples (\cite[Theorem 1.1]{HerIya11}).

In this paper, for a higher hereditary algebra, we will calculate its upper (lower) Serre dimensions, the entropy and polynomial entropy of Serre functor, and the Hochschild (co)homology entropy of Serre quasi-functor.
Given Serre functor and higher hereditary algebra have congenital relationship, the calculations become feasible.

Our main results are the following.

\medskip

{\bf Theorem A.} (= Theorem~\ref{Theorem-TFCY-Entropies}) {\it Let $A$ be a twisted $\frac{q}{p}$-Calabi-Yau algebra. Then

\medskip

{\rm (1)} the entropy of Serre functor $h_t(S)=\frac{q}{p}t$.

\medskip

{\rm (2)} the polynomial entropy of Serre functor $h^\mathrm{pol}_t(S)=0$.

\medskip

{\rm (3)} the Hochschild (co)homology entropy of Serre quasi-functor $h^{HH^\bullet}(\tilde{S})=h^{HH_\bullet}(\tilde{S})=0$ if we assume further that $A$ is elementary.

\medskip

{\rm (4)} the upper (lower) Serre dimension $\ol{\mathrm{Sdim}}A=\ul{\mathrm{Sdim}}A=\frac{q}{p}$.}

\medskip

Theorem A (1), (2) and (4) generalize the corresponding results \cite[2.6.1]{DimHaiKatKon14}, \cite[Remark 6.3]{FanFuOuc21} and \cite[Proposition 3.17]{Ela22} for fractionally Calabi-Yau algebras. To date, it is not known whether every higher representation-finite algebra, or more general, twisted fractionally Calabi-Yau algebra, is fractionally Calabi-Yau or not (\cite[Question 1.6]{ChaDarIyaMar21}). So Theorem A should have its own place.

From Theorem A and \cite[Theorem 1.1]{HerIya11}, we get immediately the following corollary.

\medskip

{\bf Corollary B.} (= Corollary~\ref{Corollary-RepFin-Entropies}) {\it Let $A$ be an indecomposable $d$-representation-finite algebra,
$r$ the number of isomorphism classes of simple $A$-modules, and $p$ the number
of indecomposable direct summands of the basic $d$-cluster tilting $A$-module. Then

\medskip

{\rm (1)} the entropy of Serre functor $h_t(S)=\frac{d(p-r)}{p}t$.

\medskip

{\rm (2)} the polynomial entropy of Serre functor $h^\mathrm{pol}_t(S)=0$.

\medskip

{\rm (3)} the Hochschild (co)homology entropy of Serre quasi-functor $h^{HH^\bullet}(\tilde{S})=h^{HH_\bullet}(\tilde{S})=0$ if we assume further that $A$ is elementary.

\medskip

{\rm (4)} the upper (lower) Serre dimension $\ol{\mathrm{Sdim}}A=\ul{\mathrm{Sdim}}A=\frac{d(p-r)}{p}$.}

\medskip

Applying the Hirzebruch-Riemann-Roch type theorem (\cite[Theorem 1]{Han20}) and Wimmer's formula (\cite[Theorem]{Wim74}), we can obtain the following Theorem C which gives the Yomdin type inequality on Hochschild homology entropy.

\medskip

{\bf Theorem C.} (= Theorem~\ref{Theorem-HHEntropy-Yomdin})
{\it Let $A$ be a finite dimensional elementary algebra of finite global dimension, $M$ a perfect $A$-bimodule complex, and $\Psi_M:=-C_MC_A^{-1}$ the dual Coxeter matrix of $M$. Then
$h^{HH_\bullet}(M) \ge \log\rho(\Psi_M)$. Here, $\rho(\Psi_M)$ is the spectral radius of the square matrix $\Psi_M$ (See Subsection 2.2).}

\medskip

I do not know whether the Gromov type inequality on Hochschild homology entropy, that is, $h^{HH_\bullet}(M) \le \log\rho(\Psi_M)$, and the Gromov and Yomdin type inequalities on Hochschild cohomology entropy, that is, $h^{HH^\bullet}(M) \le \log\rho(\Psi_M)$ and $h^{HH^\bullet}(M) \ge \log\rho(\Psi_M)$, hold or not.

The Theorem C above will be applied to show the following Theorem D.

\medskip

{\bf Theorem D.} (= Theorem~\ref{Theorem-RepInfAlg-Entropies}) {\it Let $A$ be an elementary $d$-representation-infinite algebra, and $\Phi$ the Coxeter matrix of $A$. Then

\medskip

{\rm (1)} the entropy of (inverse) Serre functor: $h_t(S)= dt+\log\rho(\Phi)$ and
$h_t(S^{-1}) \linebreak = -dt+\log\rho(\Phi^{-1}).$
Furthermore, $\rho(\Phi)=\rho(\Phi^{-1}).$

\medskip

{\rm (2)} the polynomial entropy of (inverse) Serre functor: $h^\mathrm{pol}_t(S)=s(\Phi)$ and $h^\mathrm{pol}_t(S^{-1})=s(\Phi^{-1})$. Furthermore, $s(\Phi)=s(\Phi^{-1})$.
Here, $s(\Phi)$ is the polynomial growth rate of the square matrix $\Phi$ (See Subsection 2.2).

\medskip

{\rm (3)} the Hochschild (co)homology entropy of (inverse) Serre quasi-functor: \linebreak $h^{HH^\bullet}(\tilde{S}) = h^{HH_\bullet}(\tilde{S}) = h(S) = \log\rho(\Phi) = \log\rho(\Phi^{-1}) = h(S^{-1}) = h^{HH_\bullet}(\tilde{S}^{-1}) \linebreak = h^{HH^\bullet}(\tilde{S}^{-1})$.
Here, $h(S)$ is the value $h_0(S)$ of the entropy $h_t(S)$ of the Serre functor $S$ at $t=0$ (See Subsection 2.1).

\medskip

{\rm (4)} the upper (lower) Serre dimension: $\ol{\mathrm{Sdim}}A = \ul{\mathrm{Sdim}}A = \gl A = d$.}

\medskip

Partial results of Theorem D (1), (2) and (4) for representation-infinite hereditary algebras had been obtained in \cite[Theorem 2.17]{DimHaiKatKon14}, \cite[Proposition 4.2]{Ela22} and \cite[Proposition 4.4]{FanFuOuc21}. The equalities $\rho(\Phi)=\rho(\Phi^{-1})$ and $s(\Phi)=s(\Phi^{-1})$ seem to be new. Serre functor is a categorification of Coxeter matrix. Theorem D (1), (2) and (3) suggest that, for an elementary higher representation-infinite algebra, the entropy of Serre functor and the Hochschild (co)homology entropy of Serre quasi-functor are the categorifications of spectral radius of Coxeter matrix, and polynomial entropy is the categorification of polynomial growth rate of Coxeter matrix, in some sense.

Furthermore, our main results imply that the Kikuta and Ouchi's question on the relations between entropy and Hochschild (co)homology entropy has positive answer, that is, $h^{HH^\bullet}(\tilde{S}) = h^{HH_\bullet}(\tilde{S}) = h(S)$, and the Gromov-Yomdin type equalities on entropy and Hochschild (co)homology entropy hold, that is, $h(S)=\log\rho([S])$ and $h^{HH^\bullet}(\tilde{S})= h^{HH_\bullet}(\tilde{S}) =\log\rho([H^0(\tilde{S})])$, for the Serre functor $S$ on the perfect derived category and the Serre quasi-functor $\tilde{S}$ on the perfect dg module category of an elementary twisted fractionally Calabi-Yau algebra or an indecomposable elementary higher hereditary algebra.

\medskip

The paper is structured as follows: In Section 2, we will recall the definitions of entropy, polynomial entropy, Hochschild (co)homology entropy and upper (lower) Serre dimension, and their basic properties that we need for later use. Moreover, we will reformulate Hochschild (co)homology entropy so that they are easier to manipulate in our situations. In Section 3, we will calculate these invariants for the Serre functor on the perfect derived category and the Serre quasi-functor on the perfect dg module category of a higher hereditary algebra. We will consider twisted fractionally Calabi-Yau algebras, higher representation-finite algebras, and higher representation-infinite algebras in turn. For this, we will show the Yomdin type inequality on Hochschild homology entropy for a finite dimensional elementary algebra, and prove that Hochschild cohomology entropy and Hochschild homology entropy coincide for the (inverse) Serre quasi-functor on perfect dg module category of a proper smooth dg algebra (Proposition~\ref{Proposition-HHCEntropy=HHEntropy}).

\medskip

\noindent{\bf Conventions.} Throughout this paper, $k$ is a field and $(-)^*:=\Hom_k(-,k)$ is $k$-dual. Unless stated otherwise, all algebras (resp. vector spaces, categories and functors) are $k$-algebras (resp. $k$-vector spaces, $k$-categories and $k$-functors). Moreover, all functors between triangulated categories are assumed to be triangle (= exact) functors. For a vector space complex $X$ with finite dimensional total cohomology, $\mathrm{tdim}_k(X):=\sum\limits_{i\in\mathbb{Z}} \dim_k H^i(X)\in\mathbb{Z}_{\ge 0}$ is the {\it total dimension} of the total cohomology $\bigoplus\limits_{i\in\mathbb{Z}}H^i(X)$ of $X$, and $\mathrm{sdim}_k(X):=\sum\limits_{i\in\mathbb{Z}} (-1)^i\ \!\dim_k H^i(X)\in\mathbb{Z}$ is the {\it super dimension} of $X$. For a bounded complex $X$ of finite dimensional vector spaces, $\mathrm{dim}_kX:=\sum\limits_{i\in\mathbb{Z}} \dim_k X^i\in\mathbb{Z}_{\ge 0}$ is the {\it dimension} of the underlying graded vector space $X=\bigoplus\limits_{i\in\mathbb{Z}}X^i$. For a finite dimensional algebra $A$, we denote by $\Mod A$ the category of {\it right} $A$-modules, by $\mod A$ the full subcategory of $\Mod A$ consisting of all finite dimensional right $A$-modules, by $\D(A)$ the unbounded derived category of $A$, and by $\mathcal{D}^b(A)$ the bounded derived category of $A$, that is, the full triangulated subcategory of $\D(A)$ consisting of all right $A$-module complexes with finite dimensional total cohomology.
For representation theory of algebras, we refer to \cite{AssSimSko06} and \cite{AusReiSma95}.
For knowledge of dg categories, we refer to \cite{Kel94,Kel06,Toe07,ToeVaq07,Toe11}.

\section{Entropies and Serre dimensions}

In this section, we will recall the definitions of entropy, polynomial entropy, Hochschild (co)homology entropy and upper (lower) Serre dimension, and their basic properties that we need for later use. Moreover, we will reformulate Hochschild (co)homology entropy so that they are easier to manipulate in our situations.

\subsection{Entropy}

Entropy is the asymptotically exponential growth rate of complexity.

\medskip

\noindent{\bf Complexity.}
Let ${\cal T}$ be a triangulated category with shift functor $[1]$, and $E_1,E_2$ two objects in ${\cal T}$.
The {\it complexity} of $E_2$ with respect to $E_1$ (\cite[Definition 2.1]{DimHaiKatKon14}) is the function $\dz_t(E_1,E_2): \mathbb{R}\to \mathbb{R}_{\ge 0}\cup\{\infty\}$ in the real variable $t$ given by
$$\dz_t(E_1,E_2):=\mathrm{inf}\left\{ \sum\limits_{i=1}^m e^{n_it}\ \left|\
\xymatrix@=3pt{
T_0 \ar[rr] & & T_1 \ar[dl] & \cdots & T_{m-1} \ar[rr] && T_m \ar[dl] \\
& E_1[n_1] \ar@{.>}[ul] &   &        && E_1[n_m] \ar@{.>}[ul] & }
\right.\right\}$$
where $T_0=0, T_m=E_2\oplus E'_2$ for some $E'_2\in {\cal T},$ and $T_{i-1}\to T_i\to E_1[n_i]\to,\linebreak 1\le i\le m$, are triangles in ${\cal T}$.
By convention, $\dz_t(E_1,E_2):=0$ if $E_2=0$, and $\dz_t(E_1,E_2):=\infty$ if and only if $E_2$ is not in the thick triangulated subcategory of ${\cal T}$ generated by $E_1$.
Note that $\dz_0(E_1,E_2)$ is the least number of steps required to build $E_2$ out of $\{E_1[n]\ |\ n\in\mathbb{Z}\}$.

Usually, one considers a {\it triangulated category ${\cal T}$ with a (split=classical) generator $G$}, that is, the smallest thick triangulated subcategory of ${\cal T}$ containing $G$ is ${\cal T}$ itself, or equivalently, for every object $E\in {\cal T}$, there is an object $E'\in {\cal T}$ and a tower of triangles in ${\cal T}$
$$\xymatrix@=4pt{
0=T_0 \ar[rr] && T_1 \ar[rr] \ar[dl] && T_2 \ar[dl] & \cdots & T_{m-1} \ar[rr] && T_m=E\oplus E' \ar[dl] \\
& G[n_1] \ar@{.>}[ul] && G[n_2] \ar@{.>}[ul] &&&& G[n_m] \ar@{.>}[ul] & }$$
with $m\in\mathbb{Z}_{\ge 0}$ and $n_i\in\Z$ for all $1\le i\le m$. In this case, the complexity of $E$ relative to $G$ is a function from $\mathbb{R}$ to  $\mathbb{R}_{\ge 0}$.

For some basic properties of complexity, we refer to \cite[Proposition 2.3]{DimHaiKatKon14}.

\medskip

\noindent{\bf Entropy.}
Let ${\cal T}$ be a triangulated category with a generator $G$, and $F$ an endofunctor of ${\cal T}$. The {\it entropy} of $F$ (\cite[Definition 2.5]{DimHaiKatKon14}) is the function $h_t(F):\mathbb{R}\to \mathbb{R}\cup\{-\infty\}$ in the real variable $t$ given by
$$h_t(F):=\lim\limits_{n\to\infty}\frac{1}{n}\log\dz_t(G,F^n(G)).$$
By convention, $h_t(F):=-\infty$ if $F$ is nilpotent.
It follows from \cite[Lemma 2.6]{DimHaiKatKon14} that the limit above exists in $\mathbb{R}\cup\{-\infty\}$ for every $t$ and it is independent of the choice of the generator $G$. Denote $h(F):=h_0(F)$.
If ${\cal T}\ne 0$ then $h({\rm Id}_{\cal T})=0$, but this may fail for other values of $t$.

Usually, one considers a triangulated category ${\cal T}$ admitting a dg enhancement (\cite[Definition 1.11]{CanSte17}) given by a {\it saturated dg category} $\tilde{\cal T}$ (\cite[Definition 2.4]{ToeVaq07}) which is just a triangulated dg category Morita equivalent to a proper smooth dg algebra (\cite[Proposition 5.4.2]{Toe11}), and an endofunctor $F:{\cal T}\to{\cal T}$ of ${\cal T}$ admitting a dg lift (\cite[Definition 6.7]{CanSte17}) $\tilde{F}:\tilde{\cal T}\to \tilde{\cal T}$.
In this case, entropy can be calculated by cohomology functors.

\begin{theorem} \label{Theorem-Entropy-Cohomology} {\rm (\cite[Theorem 2.7]{DimHaiKatKon14})}
Let $\mathcal{T}$ be a triangulated category admitting a dg enhancement given by a saturated dg category, and $F$ an endofunctor of $\mathcal{T}$ admitting a dg lift. Then for any generator $G$ of $\mathcal{T}$,
$$h_t(F)=\lim\limits_{n\to\infty}\frac{1}{n}\log\sum_{l\in\mathbb{Z}}\ \dim_k\Hom_{\cal T}(G,F^n(G)[l])\cdot e^{-lt}.$$
\end{theorem}

Under the circumstance of Theorem~\ref{Theorem-Entropy-Cohomology}, we have $h_t({\rm Id}_{\cal T})=0$ for all $t\in \mathbb{R}$.

\medskip

Entropy has the following properties.

\begin{lemma} \label{Lemma-Entropy-Power} {\rm (Power \cite[2.2]{DimHaiKatKon14})}
Let ${\cal T}$ be a triangulated category with a generator, $F$ an endofunctor of ${\cal T}$, and $m\in\mathbb{Z}_{>0}$. Then $h_t(F^m)=m\cdot h_t(F)$.
\end{lemma}

\begin{lemma} \label{Lemma-Entropy-Shift} {\rm (Shift \cite[Lemma 2.7]{KikShiTak20})}
Let $\mathcal{T}$ be a triangulated category admitting a dg enhancement given by a saturated dg category, $F$ an endofunctor of $\mathcal{T}$ admitting a dg lift, and $m\in\mathbb{Z}$. Then $h_t(F[m])=h_t(F)+mt$. In particular, $h_t([m])=mt$.
\end{lemma}

\begin{lemma} \label{Lemma-Entropy-Inverse} {\rm (Inverse \cite[Lemma 2.11]{FanFuOuc21})}
Let $\mathcal{T}$ be a triangulated category admitting a Serre functor and a dg enhancement given by a saturated dg category, and $F$ an autoequivalence of $\mathcal{T}$ admitting a dg lift. Then $h_t(F^{-1})=h_{-t}(F).$ In particular, $h(F^{-1})=h(F).$
\end{lemma}

\subsection{Polynomial entropy}

Polynomial entropy is the asymptotically polynomial growth rate of complexity.

\medskip

\noindent{\bf Polynomial entropy.} Let ${\cal T}$ be a triangulated category with a generator $G$, and $F$ an endofunctor of ${\cal T}$. The {\it polynomial entropy} of $F$ (\cite[Definition 2.4]{FanFuOuc21}) is the function $h_t^\mathrm{pol}(F) : \mathbb{R} \to \mathbb{R} \cup \{-\infty\}$ in the real variable $t$ given by
$$h_t^\mathrm{pol}(F) := \limsup\limits_{n\to \infty} \frac{\log \dz_t(G,F^n(G))-n\cdot h_t(F)}{\log n}.$$
It is well-defined for any $t\in\mathbb{R}$ such that $h_t(F)\ne -\infty$. In particular, it is well-defined at $t=0$ since $h_0(F)\ge 0$. Denote $h^\mathrm{pol}(F) := h_0^\mathrm{pol}(F)$.

For a triangulated category $\mathcal{T}$ admitting a dg enhancement given by a saturated dg category, and an endofunctor $F$ of $\mathcal{T}$ admitting a dg lift, the polynomial entropy can be calculated by cohomology functors.

\begin{lemma} \label{Lemma-PolyEntropy-CohFun} {\rm (\cite[Lemma 2.7]{FanFuOuc21})}
Let $\mathcal{T}$ be a triangulated category admitting a dg enhancement given by a saturated dg category, and $F$ an endofunctor of $\mathcal{T}$ admitting a dg lift. Then for any generators $G,G'$ of $\mathcal{T}$,
$$h^\mathrm{pol}_t(F)=\limsup\limits_{n\to\infty} \frac{\log\epsilon_t(G,F^n(G'))-n\cdot h_t(F)} {\log n},$$
where $\epsilon_t(X,Y):=\sum\limits_{l\in\mathbb{Z}}\dim_k\Hom_{\cal T}(X,Y[l])\cdot e^{-lt}$ for all $X,Y\in{\cal T}$ and $t\in\mathbb{R}$.
\end{lemma}

Polynomial entropy has the following properties.

\begin{lemma} \label{Lemma-PolyEntropy-Shift} {\rm (Shift \cite[Lemma 6.1]{FanFuOuc21})}
Let $\mathcal{T}$ be a triangulated category admitting a dg enhancement given by a saturated dg category, $F$ an endofunctor of $\mathcal{T}$ admitting a dg lift, and $m\in\mathbb{Z}$. Then $h_t^\mathrm{pol}(F[m]) = h_t^\mathrm{pol}(F)$.
In particular, $h_t^\mathrm{pol}([m]) = 0$.
\end{lemma}

\begin{lemma} \label{Lemma-PolyEntropy-Inverse} {\rm (Inverse \cite[Lemma 2.11]{FanFuOuc21})}
Let $\mathcal{T}$ be a triangulated category admitting a Serre functor and a dg enhancement given by a saturated dg category, and $F$ an autoequivalence of $\mathcal{T}$ admitting a dg lift. Then $h_t^\mathrm{pol}(F^{-1}) = h_{-t}^\mathrm{pol}(F)$.
In particular, $h^\mathrm{pol}(F^{-1})=h^\mathrm{pol}(F)$.
\end{lemma}

\medskip

\noindent{\bf Polynomial growth rate of a linear operator.}
Sometimes, we can reduce the polynomial entropy of an endofunctor to the polynomial growth rate of a linear operator (\cite[Proposition 4.4]{FanFuOuc21}).
Let $f$ be a non-nilpotent linear operator on a finite dimensional complex vector space $V$ endowed with some norm $\|-\|$, and $\rho(f)$ the {\it spectral radius} of $f$, that is, the maximal absolute value of eigenvalues of $f$. The {\it polynomial growth rate} of $f$ (\cite[Definition 4.1]{FanFuOuc21}) is
$$s(f) := \lim\limits_{n\to\infty} \frac{\log\|f^n\|-n\cdot \log\rho(f)}{\log n}.$$
Here, $\|f^n\|=\sup\limits_{0\ne x\in V}\frac{\|f^n(x)\|}{\|x\|}$ (See \cite[1.1.10]{BelLyu88}). As all norms on the space of matrices are equivalent, $s(f)$ is independent of the choice of the norm.

\begin{lemma} \label{Lemma-PolyEntropy-LinearOperator}
{\rm (\cite[Lemma 4.2]{FanFuOuc21})}
Let $f$ be a non-nilpotent linear operator of a finite dimensional complex vector space endowed with some norm $\|-\|$. Then  the polynomial growth rate of $f$ is well-defined, and it is precisely one less than the maximal size of the Jordan blocks whose eigenvalues are of maximal absolute value $\rho(f)$. In particular, $s(f)$ is a nonnegative integer.
\end{lemma}

\subsection{Hochschild (co)homology entropy}

Hochschild (co)homology entropy is defined not for a (triangle) endofunctor of a triangulated category but for a quasi-endofunctor of the perfect dg module category of a proper smooth dg category.

\medskip

\noindent{\bf Homotopy category of dg categories.} The category $\mathbf{dgcat}$ of (small) dg categories, whose morphisms are dg functors, is a closed symmetric monoidal category with tensor product $\otimes$ and internal Hom functor $\mathcal{H}om$. A {\it quasi-equivalence} $F$ from a dg category $\A$  to a dg category $\B$ is a dg functor $F: \A\to \B$ such that $F(x,y): \A(x,y) \to \B(F(x),F(y))$ is a quasi-isomorphism for all objects $x,y\in \A$ and the induced functor on homotopy categories $H^0(F):  H^0(\A)\to H^0(\B)$ is an equivalence.
The localization $\mathbf{hodgcat}$ of $\mathbf{dgcat}$ with respect to quasi-equivalences
is a closed symmetric monoidal category with tensor product $\otimes^L$ and internal Hom functor $\mathcal{RH}om$ (\cite[Theorem 4.5]{Kel06} and \cite[Theorem 1.3]{Toe07}).
For any two dg categories $\A$ and $\B$, let $\mathrm{rep}(\A,\B)$ be the full triangulated subcategory of the derived category
$\mathcal{D}(\A^\op\otimes \B)$ of dg $\A$-$\B$-bimodules formed by the dg $\A$-$\B$-bimodules $M$ such that the derived tensor product functor $- \otimes^L_\A M: \mathcal{D}(\A) \to \mathcal{D}(\B)$ takes the representable dg $\A$-modules to objects which are isomorphic to representable dg $\B$-modules.
Then we have a bijection
$\mathbf{hodgcat}(\A,\B) \cong \mathrm{Iso}(\mathrm{rep}(\A,\B))$
(\cite[Theorem 4.2]{Kel06} and \cite[Corollary 1.2]{Toe07}).
Let $\mathrm{rep}_\mathrm{dg}(\A,\B)$ be the full sub-dg category of the dg $\A$-$\B$-bimodule category $\Mod(\A^\op\otimes \B)$ whose objects are those of $\mathrm{rep}(\A,\B)$ which are cofibrant as dg $\A$-$\B$-bimodules. Then we have an isomorphism
$\mathcal{RH}\mathrm{om}(\A,\B) \cong \mathrm{rep}_\mathrm{dg}(\A,\B)$
in $\mathbf{hodgcat}$ (\cite[Theorem 4.5]{Kel06} and \cite[Theorem 1.3]{Toe07}).
Furthermore, we have equivalences
$H^0(\mathcal{RH}\mathrm{om}(\A,\B)) \simeq H^0(\mathrm{rep}_\mathrm{dg}(\A,\B)) \simeq \mathrm{rep}(\A,\B)$
and a bijection
$\mathbf{hodgcat}(\A,\B) \cong \mathrm{Iso}(H^0(\mathcal{RH}\mathrm{om}(\A,\B))).$ The objects in $\mathcal{RH}\mathrm{om}(\A,\B)$ or \linebreak $\mathrm{rep}_\mathrm{dg}(\A,\B)$ or $H^0(\mathcal{RH}\mathrm{om}(\A,\B))$ or $\mathrm{rep}(\A,\B)$ are called {\it quasi-functors}.

\medskip

\noindent{\bf Hochschild (co)homology entropy.}
Let $\A$ be a proper smooth dg category (\cite[Definition 2.4]{ToeVaq07}).
Denote by $\mathrm{per}_\mathrm{dg}(\A)$ the full sub-dg category of the dg category $\Mod \A$ of dg right $\A$-modules, whose objects are the cofibrant perfect dg right $\A$-modules, which is a dg enhancement of the perfect derived category $\per(\A)$ of dg right $\A$-modules.
It follows from \cite[Lemma 2.6]{ToeVaq07} that $\mathrm{per}_\mathrm{dg}(\A)$ is a saturated dg category.
Applying the formula in \cite[Page 298]{Toe11}, we obtain an isomorphism
$\mathcal{RH}\mathrm{om}(\mathrm{per}_\mathrm{dg}(\A),\mathrm{per}_\mathrm{dg}(\A)) \cong \mathrm{per}_\mathrm{dg}(\A^e)$ in $\mathbf{hodgcat}$,
where $\A^e:=\A^\op\otimes \A$ is the enveloping dg category of $\A$.
Thus we have {\it inverse Serre quasi-functor} $\tilde{S}^{-1}\in \mathcal{RH}\mathrm{om}(\mathrm{per}_\mathrm{dg}(\A),\mathrm{per}_\mathrm{dg}(\A))$ which corresponds to the {\it inverse dualizing complex} $\Theta \in \mathrm{per}_\mathrm{dg}(\A^e)$ (\cite[3.3]{Kel11}), that is, a cofibrant resolution of $\RHom_{\A^e}(\A,\A^e) \in \mathrm{per}(\A^e)$.
For any quasi-functor $\tilde{\Phi}\in \mathcal{RH}\mathrm{om}(\per_\mathrm{dg}(\A),\per_\mathrm{dg}(\A))$ and $i\in\mathbb{Z}$,
the {\it $i$-th Hochschild cohomology group} of $\tilde{\Phi}$ (\cite[Definition 2.6]{KikOuc20}) is $$HH^i(\tilde{\Phi}) := H^i(\mathcal{RH}\mathrm{om}(\per_\mathrm{dg}(\A),\per_\mathrm{dg}(\A)) (\mathrm{Id},\tilde{\Phi})),$$
and the {\it $i$-th Hochschild homology group} of $\tilde{\Phi}$ is $$HH_i(\tilde{\Phi}) := H^{-i}(\mathcal{RH}\mathrm{om}(\per_\mathrm{dg}(\A),\per_\mathrm{dg}(\A))(\tilde{S}^{-1},\tilde{\Phi})).$$
Furthermore, the {\it Hochschild cohomology entropy} of $\tilde{\Phi}$ (\cite[Definition 2.9]{KikOuc20}) is
$$h^{HH^\bullet}(\tilde{\Phi}) := \limsup\limits_{n\to\infty} \frac{1}{n} \log \sum\limits_{i\in\mathbb{Z}}\dim_kHH^i(\tilde{\Phi}^n),$$
and the {\it Hochschild homology entropy} of $\tilde{\Phi}$ is
$$h^{HH_\bullet}(\tilde{\Phi}) := \limsup\limits_{n\to\infty} \frac{1}{n} \log \sum\limits_{i\in\mathbb{Z}}\dim_kHH_i(\tilde{\Phi}^n).$$

Let $\A$ be a proper smooth dg category. A quasi-functor $\tilde{\Phi} \in \mathcal{RH}\mathrm{om}(\per_\mathrm{dg}(\A), \linebreak \per_\mathrm{dg}(\A))$ is also a quasi-functor in $H^0(\mathcal{RH}\mathrm{om}(\per_\mathrm{dg}(\A),\per_\mathrm{dg}(\A)))$, which corresponds to a morphism $\tilde{\Phi} \in \mathbf{hodgcat}(\per_\mathrm{dg}(\A), \per_\mathrm{dg}(\A))$ and further a functor $\Phi:=H^0(\tilde{\Phi}) : \per(\A) \to \per(\A)$. The following result compares the Hochschild (co)homology entropy of $\tilde{\Phi}$ with the entropy of $\Phi$.

\begin{theorem} \label{Theorem-HHEntropy-Entropy} {\rm (\cite[Theorem 2.10]{KikOuc20})}
Let $\A$ be a proper smooth dg category, $\tilde{\Phi} \in \mathcal{RH}\mathrm{om}(\per_\mathrm{dg}(\A),\per_\mathrm{dg}(\A))$ a quasi-endofunctor of $\per_\mathrm{dg}(\A)$, and $\Phi:=H^0(\tilde{\Phi})$ the induced endofunctor of $\per(\A)$. Then
$h^{HH^\bullet}(\tilde{\Phi}) \le h(\Phi)$ and $h^{HH_\bullet}(\tilde{\Phi}) \le h(\Phi).$
\end{theorem}

Furthermore, Kikuta and Ouchi posed the following question on relations between entropy and Hochschild (co)homology entropy.

\begin{question} \label{Question-Kikuta-Ouchi} {\rm (\cite[Question 2.13]{KikOuc20})
Let $\A$ be a proper smooth dg category, $\tilde{\Phi} \in \mathcal{RH}\mathrm{om}(\per_\mathrm{dg}(\A),\per_\mathrm{dg}(\A))$ a quasi-endofunctor of $\per_\mathrm{dg}(\A)$, and $\Phi:=H^0(\tilde{\Phi})$ the induced endofunctor of $\per(\A)$. When do the equalities $h^{HH^\bullet}(\tilde{\Phi}) = h^{HH_\bullet}(\tilde{\Phi}) = h(\Phi)$ hold?
}\end{question}

It was pointed by Atsushi Takahashi (\cite[Page 231]{KikOuc20}) that the equalities in Question~\ref{Question-Kikuta-Ouchi} do not always hold in general. We will provide some positive answers, cf. Remark 1 (1) and Remark 2 (3).

\medskip

\noindent{\bf Reformulation of Hochschild (co)homology entropy.}
Let $\A$ be a proper smooth dg category. Under the isomorphism $\mathcal{RH}\mathrm{om}(\mathrm{per}_\mathrm{dg}(\A),\mathrm{per}_\mathrm{dg}(\A)) \cong \mathrm{per}_\mathrm{dg}(\A^e)$ in $\mathbf{hodgcat}$, any object $\tilde{\Phi} \in \mathcal{RH}\mathrm{om}(\per_\mathrm{dg}(\A),\per_\mathrm{dg}(\A))$ corresponds to an object $M\in \per_\mathrm{dg}(\A^e)$. In particular, the identity functor $\mathrm{Id}$ on $\per_\mathrm{dg}(\A)$ corresponds to an object $\Delta\in\per_\mathrm{dg}(\A^e)$ which is a cofibrant resolution of the dg $\A$-bimodule $\A\in\per(\A^e)$.
Then for all $i\in\mathbb{Z}$ and $n\in\mathbb{Z}_{\ge 0}$, we have isomorphisms
$$\begin{aligned}
HH^i(\tilde{\Phi}^n) & = H^i(\mathcal{RH}\mathrm{om}(\mathrm{per}_\mathrm{dg}(\A),\mathrm{per}_\mathrm{dg}(\A)) (\mathrm{Id},\tilde{\Phi}^n)) \\
& \cong H^i(\mathrm{per}_\mathrm{dg}(\A^e)(\Delta,M^{\otimes_\A n})) \\
& \cong H^i(\RHom_{\A^e}(\A,M^{\otimes^L_\A n}))
\end{aligned}$$
and
$$\begin{aligned}
HH_i(\tilde{\Phi}^n) & = H^{-i}(\mathcal{RH}\mathrm{om}(\mathrm{per}_\mathrm{dg}(\A),
\mathrm{per}_\mathrm{dg}(\A))(\tilde{S}^{-1},\tilde{\Phi}^n)) \\
& \cong H^{-i}(\mathrm{per}_\mathrm{dg}(\A^e)(\Theta,M^{\otimes_\A n})) \\
& \cong H^{-i}(\RHom_{\A^e}(\RHom_{\A^e}(\A,\A^e),M^{\otimes^L_\A n})) \\
& \cong H^{-i}(M^{\otimes^L_\A n} \otimes^L_{\A^e} \RHom_{\A^e}(\RHom_{\A^e}(\A,\A^e),\A^e)) \\
& \cong H^{-i}(M^{\otimes^L_\A n} \otimes^L_{\A^e} \A) \\
& \cong H^{-i}(\A \otimes^L_{\A^e} M^{\otimes^L_\A n}).
\end{aligned}$$
So Hochschild (co)homology entropy is well-defined for any quasi-endofunctor $\tilde{\Phi}\in \mathcal{RH}\mathrm{om}(\mathrm{per}_\mathrm{dg}(\A),\mathrm{per}_\mathrm{dg}(\A))$ or $H^0(\mathcal{RH}\mathrm{om}(\mathrm{per}_\mathrm{dg}(\A),\mathrm{per}_\mathrm{dg}(\A)))$ or dg $\A$-bimodule $M\in \mathrm{per}_\mathrm{dg}(\A^e)$ or $\mathrm{per}(\A^e)$, and can be calculated by
$$h^{HH^\bullet}(\tilde{\Phi}) = h^{HH^\bullet}(M) := \limsup\limits_{n\to\infty} \frac{1}{n} \log \mathrm{tdim}_k \RHom_{\A^e}(\A,M^{\otimes^L_\A n}),$$
and
$$h^{HH_\bullet}(\tilde{\Phi}) = h^{HH_\bullet}(M) := \limsup\limits_{n\to\infty} \frac{1}{n} \log \mathrm{tdim}_k (\A \otimes^L_{\A^e} M^{\otimes^L_\A n}).$$

\subsection{Serre dimensions}

Serre dimensions are new invariants of an Ext-finite triangulated category with a generator and a Serre functor, or of a proper smooth dg algebra.

\medskip

\noindent{\bf (Inverse) Serre (quasi-)functor.} A triangulated category $\mathcal{T}$ is {\it Hom-finite} if  $\dim_k\Hom_\mathcal{T}(X,Y) <\infty$ for all $X,Y\in \mathcal{T}$.
A {\it Serre functor} (\cite[Definition 3.1]{BonKap91}) of a Hom-finite triangulated category $\mathcal{T}$ is an autoequivalence  $S:\mathcal{T}\to \mathcal{T}$ such that there is a bifunctorial isomorphism $\Hom_\mathcal{T}(X,Y)^* \cong \Hom_\mathcal{T}(Y,S(X))$.
The Serre functor, if it exists, is unique up to isomorphism.
The perfect derived category $\mathrm{per}(A)$ of a proper smooth dg algebra $A$ has Serre functors.
Indeed, $S:= -\otimes^L_AA^*$ is a Serre functor of $\mathrm{per}(A)$ with the {\it inverse Serre functor} $S^{-1}:=\RHom_A(A^*,-) \cong -\otimes^L_A\RHom_A(A^*,A)$ as a quasi-inverse.
Let $\nabla$ and $\Theta$ be any cofibrant resolutions of dg $A$-bimodule $A^*$ and $\RHom_A(A^*,A)$ respectively.
The quasi-endofunctors $\tilde{S},\tilde{S}^{-1} \in \mathcal{RH}\mathrm{om}(\mathrm{per}_\mathrm{dg}(A),\mathrm{per}_\mathrm{dg}(A))$ corresponding to  $\nabla,\Theta\in \mathrm{per}_\mathrm{dg}(A^e)$ are called {\it Serre quasi-functor} and {\it inverse Serre quasi-functor} respectively. Note that we have $\RHom_A(A^*,A) \cong \RHom_{A^e}(A,\RHom_k(A^*,A)) \linebreak \cong \RHom_{A^e}(A,A^e)$ in $\per(A^e)$. So the inverse Serre quasi-functor defined here coincides with that in 2.3.

\medskip

\noindent{\bf Upper (lower) Serre dimension.}
A triangulated category $\mathcal{T}$ is {\it Ext-finite} if  $\sum\limits_{i\in\mathbb{Z}}\dim_k\Hom_\mathcal{T}(X,Y[i])<\infty$ for all $X,Y\in \mathcal{T}$.
Let ${\cal T}$ be an Ext-finite triangulated category with a generator and a Serre functor $S:{\cal T}\to {\cal T}$. For any two generators $G,G'$ of ${\cal T}$, the {\it upper Serre dimension} of ${\cal T}$ (\cite[Definition 5.3]{ElaLun21}) is
$$\overline{\mathrm{Sdim}}{\cal T} := \limsup\limits_{n\to\infty}\frac{-\inf\{i\ |\ \Hom_{\cal T}(G,S^n(G')[i]) \ne 0\}}{n},$$
and the {\it lower Serre dimension} of ${\cal T}$ is
$$\underline{\mathrm{Sdim}}{\cal T} := \liminf\limits_{n\to\infty}\frac{-\sup\{i\ |\ \Hom_{\cal T}(G,S^n(G')[i]) \ne 0\}}{n}.$$
It follows from \cite[Lemma 6.3 and Definition 6.4]{ElaLun21} that $\overline{\mathrm{Sdim}}{\cal T}$ and $\underline{\mathrm{Sdim}}{\cal T}$ are well-defined, i.e., they are independent of the choice of the generators $G$ and $G'$.

\medskip

Let $A$ be a proper smooth dg algebra. The {\it upper Serre dimension of $A$} is $\overline{\mathrm{Sdim}}A := \overline{\mathrm{Sdim}}\ \per(A)$, and the {\it lower Serre dimension of $A$} is
$\underline{\mathrm{Sdim}}A := \underline{\mathrm{Sdim}}\ \per(A)$.
It follows from \cite[Proposition 5.5]{ElaLun21} that
$$\overline{\mathrm{Sdim}}A = \lim\limits_{n\to\infty}\frac{-\inf\ (A^*)^{\ot^L_An}}{n},$$
and
$$\underline{\mathrm{Sdim}}A = \lim\limits_{n\to\infty}\frac{-\sup\ (A^*)^{\ot^L_An}}{n},$$
where for any object $X \in \D^b(A^e)$,
$$\inf X := \inf\{i\ |\ H^i(X)\ne 0\},\quad \sup X := \sup\{i\ |\ H^i(X)\ne 0\}.$$
Moreover, the above two limits are finite.

\section{Higher hereditary algebras}

In this section, for a higher hereditary algebra, we will calculate its upper (lower) Serre dimension, the entropy and polynomial entropy of Serre functor, and Hochschild (co)homology entropy of Serre quasi-functor. Higher representation-finite algebras are twisted fractionally Calabi-Yau algebras for which these invariants can be calculated easily.

\medskip

\subsection{Twisted fractionally Calabi-Yau algebras}

Twisted fractionally Calabi-Yau algebras are a generalization of fractionally Calabi-Yau algebra, and has higher representation-finite algebras as typical examples.

\medskip

\noindent{\bf Twisted fractionally Calabi-Yau algebra.}
Let $A$ be a finite dimensional algebra of finite global dimension.
Then the {\it Nakayama functor}
$$\nu = \nu_A := \RHom_A(-,A)^*: \D^b(A) \to \D^b(A)$$
is isomorphic to the Serre functor $S=S_A:= - \otimes^L_A A^* : \D^b(A)\to\D^b(A)$.
Let $\phi$ be an algebra endomorphism of $A$. Then $\phi$ induces the restriction functor $\phi^*: \Mod A\to\Mod A, M\mapsto M_\phi$, where $M_\phi$ is the right $A$-module defined by the same vector space as $M$ but the new right $A$-module action $m\cdot a:=m\phi(a)$ for all $m\in M$ and $a\in A$. Furthermore, $\phi^*$ induces the derived functor $\phi^*:\D^b(A)\to\D^b(A)$ which is isomorphic to the derived tensor product functor $- \ot^L_A A_\phi : \D^b(A) \to \D^b(A)$, where $A_\phi$ is the twisted $A$-bimodule $A$ with the left and right $A$-module actions $b\cdot a \cdot c := ba\phi(c)$ for all $a,b,c\in A$.
A finite dimensional algebra $A$ of finite global dimension is {\it twisted fractionally Calabi-Yau and of Calabi-Yau dimension $\frac{q}{p}$} or {\it twisted $\frac{q}{p}$-Calabi-Yau} (\cite[Definition 0.3]{HerIya11}) if there exist a positive integer $p\in\mathbb{Z}_{>0}$, an integer $q\in\mathbb{Z}$, and an algebra automorphism $\phi$ of $A$, such that there is a functorial isomorphism $\nu^p \cong \phi^*[q]$ of autoequivalences on $\D^b(A)$, or equivalently, $\nu^p(A) \cong A[q]$ in $\D^b(A)$ (\cite[Proposition 4.3]{HerIya11}). When $\phi$ is the identity automorphism of $A$, $A$ is said to be {\it fractionally Calabi-Yau and of Calabi-Yau dimension $\frac{q}{p}$} or {\it $\frac{q}{p}$-Calabi-Yau}.
(Twisted) $\frac{q}{p}$-Calabi-Yau algebras are also (twisted) $\frac{qk}{pk}$-Calabi-Yau algebras for all positive integers $k$. The converse does not hold in general, but the Calabi-Yau dimension $\frac{q}{p}$ is unique as a rational number.

\medskip

\noindent{\bf Hochschild (co)homology entropy of (inverse) Serre quasi-functor.} The following result implies that Hochschild cohomology entropy coincides with Hochschild homology entropy for the (inverse) Serre quasi-functor on the perfect dg module category of a proper smooth dg algebra.

\begin{proposition} \label{Proposition-HHCEntropy=HHEntropy}
Let $A$ be a proper smooth dg algebra, and $\tilde{S}$ the Serre quasi-functor on the perfect dg module category $\per_\mathrm{dg}(A)$ of $A$. Then
$h^{HH^\bullet}(\tilde{S})=h^{HH_\bullet}(\tilde{S})$ and $h^{HH^\bullet}(\tilde{S}^{-1})=h^{HH_\bullet}(\tilde{S}^{-1})$.
\end{proposition}

\begin{proof} Since $A$ is a proper smooth dg algebra, we have isomorphisms
$$\RHom_A(A^*,A) \cong \RHom_{A^e}(A,\RHom_k(A^*,A)) \cong \RHom_{A^e}(A,A^e)$$
and $A^*\otimes^L_A\RHom_A(A^*,A) \cong A$ in $\per(A^e)$.
Then we get isomorphisms
$$\begin{aligned}\RHom_{A^e}(A,(A^*)^{\otimes^L_An}) & \cong (A^*)^{\otimes^L_An}\otimes^L_{A^e}\RHom_{A^e}(A,A^e) \\ & \cong (A^*)^{\otimes^L_An}\otimes^L_{A^e}\RHom_A(A^*,A) \\ & \cong A\otimes^L_{A^e}(A^*)^{\otimes^L_A{n-1}}\end{aligned}$$
in $\D^b(k)$ for all $n\in\mathbb{Z}_{>0}$. Thus
$$\begin{aligned}
h^{HH^\bullet}(\tilde{S}) & = \limsup\limits_{n\to\infty}\frac{1}{n}\log \mathrm{tdim}_k \RHom_{A^e}(A,(A^*)^{\otimes^L_An}) \\
& = \limsup\limits_{n\to\infty}\frac{1}{n}\log \mathrm{tdim}_k(A\otimes^L_{A^e}(A^*)^{\otimes^L_A{n-1}}) \\
& = h^{HH_\bullet}(\tilde{S}).
\end{aligned}$$

Next, we have isomorphisms
$$\begin{aligned}
\RHom_{A^e}(A,\RHom_A(A^*,A)^{\otimes^L_An}) & \cong \RHom_A(A^*,A)^{\otimes^L_An} \otimes^L_{A^e}\RHom_{A^e}(A,A^e) \\
& \cong \RHom_A(A^*,A)^{\otimes^L_An} \otimes^L_{A^e}\RHom_A(A^*,A) \\
& \cong A\otimes^L_{A^e}\RHom_A(A^*,A)^{\otimes^L_A{n+1}}
\end{aligned}$$
in $\D^b(k)$ for all $n\in\mathbb{Z}_{>0}$. Thus
$$\begin{aligned}
h^{HH^\bullet}(\tilde{S}^{-1}) & = \limsup\limits_{n\to\infty}\frac{1}{n}\log\mathrm{tdim}_k \RHom_{A^e}(A,\RHom_A(A^*,A)^{\otimes^L_An}) \\
& = \limsup\limits_{n\to\infty}\frac{1}{n}\log\mathrm{tdim}_k (A\otimes^L_{A^e}\RHom_A(A^*,A)^{\otimes^L_A{n+1}}) \\
& = h^{HH_\bullet}(\tilde{S}^{-1}).
\end{aligned}$$

Now we have finished the proof of Proposition~\ref{Proposition-HHCEntropy=HHEntropy}.
\end{proof}

\medskip

\noindent{\bf Entropies and Serre dimensions for twisted fractionally Calabi-Yau algebras.} Our first main result is the following, in which (1), (2) and (4) generalize the corresponding results \cite[2.6.1]{DimHaiKatKon14}, \cite[Remark 6.3]{FanFuOuc21} and \cite[Proposition 3.17]{Ela22} for fractionally Calabi-Yau algebras.

\begin{theorem} \label{Theorem-TFCY-Entropies}
Let $A$ be a twisted $\frac{q}{p}$-Calabi-Yau algebra. Then

\medskip

{\rm (1)} the entropy of Serre functor $h_t(S)=\frac{q}{p}t$.

\medskip

{\rm (2)} the polynomial entropy of Serre functor $h^\mathrm{pol}_t(S)=0$.

\medskip

{\rm (3)} the Hochschild (co)homology entropy of Serre quasi-functor $h^{HH^\bullet}(\tilde{S})=h^{HH_\bullet}(\tilde{S})=0$ if we assume further that $A$ is elementary.

\medskip

{\rm (4)} the upper (lower) Serre dimension $\ol{\mathrm{Sdim}}A=\ul{\mathrm{Sdim}}A=\frac{q}{p}$.
\end{theorem}

\begin{proof}
(1) Since $A$ is a twisted $\frac{q}{p}$-Calabi-Yau algebra, there is an algebra automorphism $\phi$ of $A$ such that $S^p\cong \phi^*[q]$ as autoequivalences of $\per(A)$.
In view of $A_{\phi^n}\in\mod A$, we have
$$\begin{array}{ll}
h_t(\phi^*) & =\lim\limits_{n\to\infty}\frac{1}{n}\log\sum\limits_{l\in\mathbb{Z}}\ \dim_k\Hom_{\D^b(A)}(A,(\phi^*)^n(A)[l])\cdot e^{-lt} \\ [4mm]
& = \lim\limits_{n\to\infty}\frac{1}{n}\log\sum\limits_{l\in\mathbb{Z}}\ \dim_k\Hom_{\D^b(A)}(A,A_{\phi^n}[l])\cdot e^{-lt} \\ [4mm]
& = \lim\limits_{n\to\infty}\frac{1}{n}\log \dim_kA_{\phi^n}
= \lim\limits_{n\to\infty}\frac{1}{n}\log \dim_kA = 0.
\end{array}$$
Furthermore, from Lemma~\ref{Lemma-Entropy-Power} and Lemma~\ref{Lemma-Entropy-Shift}, we obtain $p\cdot h_t(S) = h_t(S^p) = h_t(\phi^*[q]) =  h_t(\phi^*)+qt = qt$. Thus $h_t(S)=\frac{q}{p}t$.

\medskip

(2) Analogous to \cite[Remark 6.3]{FanFuOuc21}.
Since $A$ is twisted $\frac{q}{p}$-Calabi-Yau, we have $S^p(A) \cong A[q]$ in $\D^b(A)$  (\cite[Proposition 4.3]{HerIya11}).
For any $n\in\mathbb{Z}_{>0}$, we can write $n=ap+b$ with $a,b\in\mathbb{Z}$ and $0\le b\le p-1$.
Then we get
$$\epsilon_t(A,S^n(A)) = \epsilon_t(A,S^b(A)[qa]) = \epsilon_t(A,S^b(A))\cdot e^{qat}.$$
Furthermore, from (1), we obtain
$$\log\epsilon_t(A,S^n(A)) - n\cdot h_t(S) = \log\epsilon_t(A,S^b(A)) - \frac{qb}{p}t.$$
Since the absolute value of the right hand side is bounded by
$$\max\limits_{0\le b\le p-1}\{|\log\epsilon_t(A,S^b(A))|\}+|qt|,$$
which is independent of $n$, we have $h_t^\mathrm{pol}(S)=0$.

\medskip

(3) As in the proof of Proposition~\ref{Proposition-HHCEntropy=HHEntropy}, we have  $\RHom_{A^e}(A,(A^*)^{\otimes^L_An}) \cong A\otimes^L_{A^e}(A^*)^{\otimes^L_A{n-1}}$ in $\D^b(k)$ for all $n\in\mathbb{Z}_{>0}$.
For any $n\in\mathbb{Z}_{>0}$, we can write $n-1=ap+b$ with $a,b\in\mathbb{Z}$ and $0\le b\le p-1$.
Thus
$$\begin{array}{ll}
\RHom_{A^e}(A,(A^*)^{\otimes^L_An})
& \cong A\otimes^L_{A^e}(A^*)^{\otimes^L_A{n-1}} \\
& \cong A\otimes^L_{A^e}((A^*)^{\otimes^L_Ab}\otimes^L_A ((A^*)^{\otimes^L_Ap})^{\otimes^L_Aa}) \\
& \cong A\otimes^L_{A^e}((A^*)^{\otimes^L_Ab}\otimes^L_A (A_{\phi}[q])^{\otimes^L_Aa}) \\
& \cong A\otimes^L_{A^e}((A^*)^{\otimes^L_Ab}\otimes^L_AA_{\phi^a}[qa]) \\
& \cong A_{\phi^a}\otimes^L_{A^e}(A^*)^{\otimes^L_Ab}[qa],
\end{array}$$
where the third isomorphism holds since $A$ is twisted Calabi-Yau,
and the fourth isomorphism follows from $(A_{\phi})^{\otimes^L_Aa} \cong (A_{\phi})^{\otimes_Aa} \cong A_{\phi^a}$ in $\D^b(A^e)$.
Let $P$ be a minimal projective resolution of $A$-bimodule $A^*$.
Then $\mathrm{tdim}_k(A_{\phi^a}\otimes^L_{A^e}(A^*)^{\otimes^L_Ab}[qa]) = \mathrm{tdim}_k(A_{\phi^a}\otimes^L_{A^e}(A^*)^{\otimes^L_Ab}) = \mathrm{tdim}_k(A_{\phi^a}\otimes_{A^e}P^{\otimes_Ab})$.
Since $A$ is a finite dimensional elementary algebra, the factor algebra $A/\rad A$ of $A$ modulo its Jacobson radical $\rad A$ is a direct sum of finitely many copies of $k$, in particular, $A/\rad A$ is separable.
By the assumption that $A$ is twisted Calabi-Yau, we know $A$ is of finite global dimension.
Applying \cite[Corollary 18]{EilRosZel57}, we obtain $\gl A^e<\infty$, in particular, $\pd_{A^e}A<\infty$, and further, $A$ is (homologically) smooth.
Thus, as a graded vector space, $P$ is finite dimensional.
Since the $i$-th component of the complex $A_{\phi^a}\otimes_{A^e}P^{\otimes_Ab}$ is a quotient vector space of the $i$-th component of the complex $A_{\phi^a}\otimes_kP^{\otimes_Ab}$ for all $i\in\mathbb{Z}$, we have $\mathrm{dim}_k(A_{\phi^a}\otimes_{A^e}P^{\otimes_Ab})\le\mathrm{dim}_k(A_{\phi^a}\otimes_kP^{\otimes_Ab})$.  So $\mathrm{tdim}_k(A_{\phi^a}\otimes_{A^e}P^{\otimes_Ab})$ is bounded by $\mathrm{dim}_k(A_{\phi^a}\otimes_{A^e}P^{\otimes_Ab})$, then by $\mathrm{dim}_k(A_{\phi^a}\otimes_kP^{\otimes_kb}) = \mathrm{dim}_k(A\otimes_kP^{\otimes_kb})$, and further by $\max\limits_{0\le b\le p-1}\{\mathrm{dim}_k(A\otimes_kP^{\otimes_kb})\}$ which is finite and independent of $n$.
Hence $h^{HH^\bullet}(\tilde{S}) = \limsup\limits_{n\to\infty}\frac{1}{n}\log \mathrm{tdim}_k \RHom_{A^e}(A,(A^*)^{\otimes^L_An}) =0$.
By Proposition~\ref{Proposition-HHCEntropy=HHEntropy},
we have $h^{HH_\bullet}(\tilde{S}) = h^{HH^\bullet}(\tilde{S}) = 0$.

\medskip

(4) For any $n\in\mathbb{Z}_{>0}$, we have $S^{np}(A) \cong A[nq]$ in $\D^b(A)$. Since $A[nq]$ is a stalk complex concentrating on degree $-nq$, we get $\sup S^{np}(A) = \inf S^{np}(A) = -nq$.
Furthermore, we obtain $\ol{\mathrm{Sdim}}A = \ul{\mathrm{Sdim}}A = \lim\limits_{n\to \infty}\frac{nq}{np} = \frac{q}{p}$.
\end{proof}

\begin{remark}{\rm
(1) By Theorem~\ref{Theorem-TFCY-Entropies} (1) and (3), we have $h^{HH^\bullet}(\tilde{S}) = h^{HH_\bullet}(\tilde{S}) = h(S)=0$, which implies that the Kikuta and Ouchi's question (Question~\ref{Question-Kikuta-Ouchi}) has positive answer for the Serre quasi-functor on perfect dg module category of an elementary twisted fractionally Calabi-Yau algebra.

(2) For an elementary twisted $\frac{q}{p}$-Calabi-Yau algebra $A$, by \cite[Proposition 4.3]{HerIya11}, we have $S^p(A) \cong A[q]$ in $\D^b(A)$. So the functor $S^p[-q]$ sends an indecomposable projective right $A$-module to an indecomposable projective right $A$-module. Thus under the basis of the Grothendieck group $K_0(\D^b(A))$ of $\D^b(A)$ formed by the isomorphism classes of indecomposable projective right $A$-modules, the matrix of $(-1)^q[S]^p=[S^p[-q]]: K_0(\D^b(A)) \to K_0(\D^b(A)),\ [X]\mapsto [S^p(X)[-q]]$, is a permutation matrix, and thus an orthogonal matrix or a matrix of finite order. Hence the eigenvalues of $(-1)^q[S]^p, [S]^p$ and $[S]$, are roots of unity. Furthermore, the spectral radius of $[S]$ is equal to 1, that is, $\log\rho([S])=0$. So the Gromov-Yomdin type equality on entropy $h(S)=\log\rho([S])=0$ holds for the Serre functor on the perfect derived category, and the Gromov-Yomdin type equality on Hochschild (co)homology entropy $h^{HH_\bullet}(\tilde{S}) = h^{HH^\bullet}(\tilde{S}) = \log\rho([H^0(\tilde{S})]) = \log\rho([S]) =0$ holds for the Serre quasi-functor on the perfect dg module category, of an elementary twisted fractionally Calabi-Yau algebra.
}\end{remark}

\subsection{Higher representation-finite algebras}

Higher representation-finite algebras are generalizations of representation-finite hereditary algebras, which are typical examples of twisted fractionally Calabi-Yau algebras.

\medskip

\noindent{\bf Higher representation-finite algebras.} Let $d\in\mathbb{Z}_{>0}$ be a positive integer. A finite dimensional algebra $A$ is {\it $d$-representation-finite}
(\cite[Definition 2.2]{IyaOpp11}) if $\gl A\le d$ and there exists a {\it $d$-cluster tilting $A$-module} $M$, that is, an $A$-module $M\in\mod A$ satisfying
$$\begin{array}{ll}
\add M & = \{X\in\mod A\ |\ \Ext_A^i(M,X)=0, \forall\ 1\le i\le d-1\} \\
& =  \{X\in\mod A\ |\ \Ext_A^i(X,M)=0, \forall\ 1\le i\le d-1\}.
\end{array}$$
Here, $\add M$ denotes the full subcategory of $\mod A$ consisting of all direct summands of direct sums of finite copies of $M$.

1-representation-finite algebras are just representation-finite hereditary algebras by definition.
2-representation-finite algebras are exactly the truncated Jacobian algebras of selfinjective quivers with potentials (\cite[Theorem 3.11]{HerIya112}).
Moreover, one can construct many $d$-representation-finite algebras by tensor product (\cite[Corollary 1.5]{HerIya11}) and higher APR tilting (\cite[Theorem 4.2 and Theorem 4.7]{IyaOpp11}).

\begin{theorem} \label{Theorem-dRepFin-TFCY} {\rm (\cite[Theorem 1.1]{HerIya11})}
Let $A$ be an indecomposable $d$-representation-finite algebra.
Then $A$ is twisted $\frac{d(p-r)}{p}$-Calabi-Yau, where $r$ is the number of isomorphism classes of simple $A$-modules and $p$ is the number of indecomposable direct summands of the basic $d$-cluster tilting $A$-module.
\end{theorem}

\medskip

\noindent{\bf Entropies and Serre dimensions for higher representation-finite algebras.} Applying Theorem~\ref{Theorem-dRepFin-TFCY} and Theorem~\ref{Theorem-TFCY-Entropies}, we can obtain immediately the following corollary.

\begin{corollary} \label{Corollary-RepFin-Entropies}
Let $A$ be an indecomposable $d$-representation-finite algebra,
$r$ the number of isomorphism classes of simple $A$-modules, and $p$ the number
of indecomposable direct summands of the basic $d$-cluster tilting $A$-module. Then

\medskip

{\rm (1)} the entropy of Serre functor $h_t(S)=\frac{d(p-r)}{p}t$.

\medskip

{\rm (2)} the polynomial entropy of Serre functor $h^\mathrm{pol}_t(S)=0$.

\medskip

{\rm (3)} the Hochschild (co)homology entropy of Serre quasi-functor $h^{HH^\bullet}(\tilde{S})=h^{HH_\bullet}(\tilde{S})=0$ if we assume further that $A$ is elementary.

\medskip

{\rm (4)} the upper (lower) Serre dimension $\ol{\mathrm{Sdim}}A=\ul{\mathrm{Sdim}}A=\frac{d(p-r)}{p}$.
\end{corollary}

\subsection{Higher representation-infinite algebras}

All kinds of entropies and upper (lower) Serre dimension for an elementary higher representation-infinite algebra are completely determined by its global dimension and the linear invariants of its Coxeter matrix.

\medskip

\noindent{\bf Coxeter matrix.} Let $A$ be a finite dimensional elementary algebra, or equivalently, a bound quiver algebra $kQ/I$ where $Q$ is a finite quiver and $I$ is an admissible ideal of the path algebra $kQ$ (\cite{AssSimSko06,AusReiSma95}). Let $e_1,\cdots,e_r$ be a complete set of orthogonal primitive idempotents of $A$. The {\it Cartan matrix} of $A$ is the $r\times r$ integer-valued matrix $C_A:=(c_{ij})$ where $c_{ij}:= \dim_k\Hom_A(e_iA,e_jA) = \dim_ke_jAe_i$ for all $1\le i,j\le r$. Then the $i$-th row of $C_A$ is the dimension vector of the indecomposable projective left $A$-module $Ae_i$ and the $j$-th column of $C_A$ is the dimension vector of the indecomposable projective right $A$-module $e_jA$ for all $1\le i,j\le r$.

If $A$ is of finite global dimension then the Cartan determinant $\det C_A$ of $A$ is $\pm 1$ (\cite{Eil54}). Thus the Cartan matrix $C_A$ of $A$ is invertible and its inverse matrix $C_A^{-1}$ is also integer-valued.
Denote by $K_0(\mathcal{D}^b(A))$ the Grothendieck group of $\mathcal{D}^b(A)$. Then the Serre functor $S=-\otimes^L_AA^*$ induces the group automorphism
$$[S] : K_0(\mathcal{D}^b(A)) \to K_0(\mathcal{D}^b(A)),\ [X] \mapsto [S(X)].$$
Since $S(e_jA)=(Ae_j)^*$ for all $1\le j\le r$, the matrix of $[S]$ under the basis of $K_0(\mathcal{D}^b(A))$ consisting of the isomorphism classes of simple $A$-modules is $C^T_A C_A^{-1}$ by \cite[Theorem 1 (1)]{Han20}. The {\it Coxeter matrix} of $A$ is the $r\times r$ integer-valued matrix $\Phi=\Phi_A := -C^T_A\cdot C^{-1}_A$ (See \cite[Chapter III, Definition 3.14]{AssSimSko06}). Then we have the following commutative diagram:
$$\xymatrix{
K_0(\mathcal{D}^b(A)) \ar[r]^-{[S]} \ar[d]^-{\cong}& K_0(\mathcal{D}^b(A))\ar[d]^-{\cong} & [X] \ar@{|->}[rr] \ar@{|->}[d] && [S(X)] \ar@{|->}[d] \\
\mathbb{Z}^r \ar[r]^-{-\Phi} & \mathbb{Z}^r & \underline{\dim}X \ar@{|->}[r] & -\Phi\cdot\underline{\dim}X \ar@{=}[r] & \underline{\dim}S(X),
}$$
where $\underline{\dim}X$ is the dimension vector of the bounded complex $X$ of $A$-modules (\cite[Page 93]{Han20}).
More general, the {\it Cartan matrix} of a finite dimensional $A$-bimodule $M$ is the $r\times r$ integer-valued matrix $C_M:=(c_{ij})$ where $c_{ij}:= \dim_k\Hom_A(e_iA,e_jM) = \dim_ke_jMe_i$ for all $1\le i,j\le r$. For any $A$-bimodule complex $M$ of finite dimensional total cohomology, its {\it Cartan matrix} is the $r\times r$ integer-valued matrix $C_M := \sum\limits_{l\in\Z}(-1)^l\ C_{H^l(M)}$ (\cite[Remark 3]{Han20}), its {\it Coxeter matrix} is the $r\times r$ integer-valued matrix $\Phi_M:=-C_M^TC_A^{-1}$, and its {\it dual Coxeter matrix} is the $r\times r$ integer-valued matrix $\Psi_M:=-C_MC_A^{-1}$. By \cite[Lemma 5]{Han20}, we have $\Psi_M=-C_MC_A^{-1} =-C^T_{M^*}C_A^{-1} = \Phi_{M^*}$.

\medskip

\noindent{\bf Yomdin type inequality on Hochschild homology entropy.} We need the following Wimmer's formula to calculate Hochschild (co)homology entropy.

\begin{lemma} \label{Lemma-Trace-SpectralRadius} {\rm (\cite[Theorem]{Wim74})}
Let $M\in M_m(\mathbb{C})$ be an $m\times m$ complex matrix, and $\rho(M)$ the spectral radius of $M$. Then
$$\limsup\limits_{n\to\infty} \sqrt[n]{|\tr(M^n)|}=\rho(M).$$
\end{lemma}

The following result gives the Yomdin type inequality on Hochschild homology entropy. I do not know whether the Gromov type inequality on Hochschild homology entropy, that is, $h^{HH_\bullet}(M) \le \log\rho(\Psi_M)$, and the Gromov and Yomdin type inequalities on Hochschild cohomology entropy, i.e., $h^{HH^\bullet}(M) \le \log\rho(\Psi_M)$ and $h^{HH^\bullet}(M) \ge \log\rho(\Psi_M)$, hold or not.

\begin{theorem} \label{Theorem-HHEntropy-Yomdin} 
Let $A$ be a finite dimensional elementary algebra of finite global dimension, $M$ a bounded $A$-bimodule complex, and $\Psi_M:=-C_MC_A^{-1}$ the dual Coxeter matrix of $M$. Then
$h^{HH_\bullet}(M) \ge \log\rho(\Psi_M).$
\end{theorem}

\begin{proof}
Applying \cite[Theorem 1 (4) and (2)]{Han20} and Lemma~\ref{Lemma-Trace-SpectralRadius}, we obtain
$$\begin{array}{ll}
& \limsup\limits_{n\to\infty}\ (\mathrm{tdim}_k(A\otimes^L_{A^e}M^{\otimes^L_An}))^{\frac{1}{n}} \\ [4mm]
\ge & \limsup\limits_{n\to\infty}\ |\mathrm{sdim}_k(A\otimes^L_{A^e}M^{\otimes^L_An})|^{\frac{1}{n}} \\ [4mm]
\stackrel{(4)}{=} & \limsup\limits_{n\to\infty}\ |\tr(C^{-1}_A\cdot C_{M^{\otimes^L_An}})|^{\frac{1}{n}} \\ [4mm]
\stackrel{(2)}{=} & \limsup\limits_{n\to\infty}\ |\tr(C^{-1}_A\cdot C_M(C^{-1}_A C_M)^{n-1})|^{\frac{1}{n}} \\ [4mm]
= & \limsup\limits_{n \to\infty}\ |\tr((C^{-1}_A C_M)^n)|^{\frac{1}{n}} \\ [4mm]
\stackrel{L8}{=} & \rho(C^{-1}_AC_M )= \rho(C_M C^{-1}_A)=  \rho(\Psi_M).
\end{array}$$
Thus
$h^{HH_\bullet}(M) = \limsup\limits_{n\to\infty}\frac{1}{n}\log
\mathrm{tdim}_k(A \otimes^L_{A^e}M^{\otimes^L_An}) \ge \log\rho(\Psi_M).$
\end{proof}
\medskip

\noindent{\bf Higher representation-infinite algebras.}
Let $A$ be a finite dimensional algebra of finite global dimension.
For any integer $d\in\mathbb{Z}$, the {\it $d$-Nakayama functor} $\nu_d:=\RHom_A(-,A)^*[-d]$, or naturally isomorphically, {\it $d$-Serre functor} $S_d:=-\otimes^L_AA^*[-d]$, is an autoequivalence on the bounded derived category $\mathcal{D}^b(A)$ of $A$ with the quasi-inverse $\nu^{-1}_d=S^{-1}_d=\RHom_A(A^*[-d],-)$.

Let $d$ be a positive integer.
A finite dimensional algebra $A$ is {\it $d$-representation-infinite} (\cite[Definition 2.7]{HerIyaOpp14}) if $\gl A\le d$ and $\nu_d^{-n}(P) \in \mod A$ for any indecomposable projective right $A$-module $P$ and $n\in\mathbb{Z}_{\ge 0}$.
In this case, $\gl A=d$ since $\Ext^d_A(A^*,A) \cong \nu_d^{-1}(A) \ne 0$.

1-representation-infinite algebras are just representation-infinite hereditary algebras by definition. Moreover, one can construct many $d$-representation-infinite algebras by tensor product (\cite[Theorem 2.10]{HerIyaOpp14}) and higher APR-tilting (\cite[Theorem 2.13]{HerIyaOpp14} and \cite[Theorem 3.1]{MizYam16}).

\medskip

\noindent{\bf Higher hereditary algebras.}
Let $d$ be a positive integer. A finite dimensional algebra $A$ is {\it $d$-hereditary} (\cite[Definition 3.2]{HerIyaOpp14}) if $\gl A\le d$ and $\nu_d^n(A) \in \D^{d\Z}(A)$ for all $n\in\Z$, where $\D^{d\Z}(A) := \{X\in\D^b(A)\ |\ H^i(X)=0, \forall i\in\Z\backslash d\Z\}$.
\medskip

The following result is the dichotomy theorem of higher hereditary algebras.

\begin{theorem} {\rm (\cite[Theorem 3.4]{HerIyaOpp14})}
Let $A$ be an indecomposable finite dimensional algebra. Then $A$ is $d$-hereditary if and only if it is either $d$-representation-finite or $d$-representation-infinite.
\end{theorem}

\medskip

\noindent{\bf Entropies and Serre dimensions for higher representation-infinite algebras.}
Now we apply the Yomdin type inequality on Hochschild homology entropy in Theorem~\ref{Theorem-HHEntropy-Yomdin} to show the following theorem.

\begin{theorem} \label{Theorem-RepInfAlg-Entropies}
Let $A$ be an elementary $d$-representation-infinite algebra, and $\Phi$ the Coxeter matrix of $A$. Then

\medskip

{\rm (1)} the entropy of (inverse) Serre functor: $h_t(S)= dt+\log\rho(\Phi)$ and
$h_t(S^{-1}) \linebreak = -dt+\log\rho(\Phi^{-1}).$
Furthermore, $\rho(\Phi)=\rho(\Phi^{-1}).$

\medskip

{\rm (2)} the polynomial entropy of (inverse) Serre functor: $h^\mathrm{pol}_t(S)=s(\Phi)$ and $h^\mathrm{pol}_t(S^{-1})=s(\Phi^{-1})$. Furthermore, $s(\Phi)=s(\Phi^{-1})$.

\medskip

{\rm (3)} the Hochschild (co)homology entropy of (inverse) Serre quasi-functor: \linebreak $h^{HH^\bullet}(\tilde{S}) = h^{HH_\bullet}(\tilde{S}) = h(S) = \log\rho(\Phi) = \log\rho(\Phi^{-1}) = h(S^{-1}) = h^{HH_\bullet}(\tilde{S}^{-1}) \linebreak = h^{HH^\bullet}(\tilde{S}^{-1})$.

\medskip

{\rm (4)} the upper (lower) Serre dimension: $\ol{\mathrm{Sdim}}A = \ul{\mathrm{Sdim}}A = \gl A = d$.
\end{theorem}

\begin{proof}
(1) Let $P_1,\cdots,P_r$ be a complete set of indecomposable projective right $A$-modules.
Then the isomorphism classes $[P_1],\cdots,[P_r]$ form a $\mathbb{Z}$-basis of the Grothendieck group $K_0(\D^b(A))$.
The {\it Euler form} of $A$ is the $\mathbb{Z}$-bilinear form
$$\chi: K_0(\D^b(A)) \times K_0(\D^b(A)) \to \mathbb{Z},\ ([X],[Y]) \mapsto \sum\limits_{l\in\mathbb{Z}}(-1)^l\ \! \dim_k\Ext^l_A(X,Y).$$
The {\it Ringel form} of $A$ is the $\mathbb{Z}$-bilinear form
$$\langle -,-\rangle_A : \mathbb{Z}^r \times \mathbb{Z}^r \to \mathbb{Z},\ (x,y) \mapsto x^TC^{-T}_Ay.$$
By the Hirzebruch-Riemann-Roch type formula \cite[Theorem 1 (1)]{Han20}, we have
$$\chi([X],[Y])=\langle \underline{\dim} X,\underline{\dim} Y\rangle_A.$$
Thus $\chi$ is non-degenerate. Let the $\mathbb{R}$-vector space
$$K_0(\D^b(A))_{\mathbb{R}} := K_0(\D^b(A)) \otimes_{\mathbb{Z}} \mathbb{R}$$
be the scalar extension of $K_0(\D^b(A))$ and the $\mathbb{R}$-bilinear form
$$\chi_{\mathbb{R}}: K_0(\D^b(A))_{\mathbb{R}} \times K_0(\D^b(A))_{\mathbb{R}} \to \mathbb{R}$$
the scalar extension of $\chi$. Then $\chi_{\mathbb{R}}$ is also non-degenerate.

For any linear operator $f: K_0(\D^b(A))_\mathbb{R} \to K_0(\D^b(A))_\mathbb{R}$,
we define
$$\|f\| := \sum\limits_{i,j=1}^r|\chi_\mathbb{R}([P_i],f([P_j]))|.$$
Since $\chi_\mathbb{R}$ is non-degenerate, $\|-\|$ is a norm on the space $\End_{\mathbb{R}}(K_0(\D^b(A))_\mathbb{R})$ of linear operators on $K_0(\D^b(A))_\mathbb{R}$.

Since $A$ is $d$-representation-infinite, $\nu_d^{-n}(P_i)$ is isomorphic to an indecomposable $A$-module for all $n\in\mathbb{Z}_{\ge 0}$ and $1\le i\le r$. Thus
$$\begin{array}{lll}
h_t(\nu_d^{-1})
& = \lim\limits_{n\to\infty} \frac{1}{n} \log \sum\limits_{l\in\mathbb{Z}}\dim_k\Ext^l_A(A,\nu_d^{-n}(A))\cdot e^{-lt} & \mbox{(Theorem~\ref{Theorem-Entropy-Cohomology})} \\ [4mm]
& = \lim\limits_{n\to\infty} \frac{1}{n} \log \sum\limits_{i,j=1}^r\dim_k\Hom_A(P_i,\nu_d^{-n}(P_j)) & (\nu_d^{-n}(P_j)\in\mod A) \\ [4mm]
& = \lim\limits_{n\to\infty} \frac{1}{n} \log \sum\limits_{i,j=1}^r |\chi_\mathbb{R}([P_i],[\nu_d^{-1}]^n([P_j]))| & (\nu_d^{-n}(P_j)\in\mod A) \\ [4mm]
& = \lim\limits_{n\to\infty}\frac{1}{n}\log\|[\nu_d^{-1}]^n\| \\ [4mm]
& = \log\rho([\nu_d^{-1}]) & \mbox{(Gelfand formula)} \\ [4mm]
& = \log\rho([\nu^{-1}])= \log\rho([S^{-1}]) = \log\rho(\Phi^{-1}). & ([\nu_d^{-1}]=(-1)^d[\nu^{-1}]) \\ [4mm]
\end{array}$$
Here, the Gelfand formula is $\rho(A)=\lim\limits_{n\to\infty} \|A^n\|^{\frac{1}{n}}$ for all complex square matrix $A$ (See \cite[1.3.3]{BelLyu88}). By Lemma~\ref{Lemma-Entropy-Shift}, we have $h_t(\nu^{-1})+dt = h_t(\nu^{-1}_d) = \log\rho(\Phi^{-1})$.
Thus $h_t(\nu^{-1}) = -dt+\log\rho(\Phi^{-1})$.

\medskip

Similarly, since $A$ is $d$-representation-infinite, by \cite[Proposition 2.9]{HerIyaOpp14},
we have $\nu_d(A)=A^*[-d]$ and $\nu_d^n(A^*)\in\mod A$ for all $n\in\mathbb{Z}_{\ge 0}$.
Let $I_1,\cdots, I_r$ be a complete set of indecomposable injective right $A$-modules.
Then $\nu_d^n(I_j)\in\mod A$ for all $n\in\mathbb{Z}_{\ge 0}$ and $1\le j\le r$. Thus
$$\begin{array}{lll}
 & h_t(\nu_d) \\
= & \lim\limits_{n\to\infty} \frac{1}{n} \log \sum\limits_{l\in\mathbb{Z}}\dim_k\Ext^l_A(A,\nu_d^n(A))\cdot e^{-lt} & \mbox{(Theorem~\ref{Theorem-Entropy-Cohomology})} \\ [4mm]
= & \lim\limits_{n\to\infty} \frac{1}{n} \log \sum\limits_{l\in\mathbb{Z}}\dim_k\Ext^{l-d}_A(A,\nu_d^{n-1}(A^*))\cdot e^{-lt} & (\nu_d(A)=A^*[-d])\\ [4mm]
= & \lim\limits_{n\to\infty} \frac{1}{n} \log \sum\limits_{q\in\mathbb{Z}}\dim_k\Ext^q_A(A,\nu_d^{n-1}(A^*))\cdot e^{-qt} & (e^{-dt}\mbox{ is independent of } n)\\ [4mm]
= & \lim\limits_{n\to\infty} \frac{1}{n} \log \sum\limits_{i,j=1}^r\dim_k\Hom_A(P_i,\nu_d^{n-1}(I_j)) & [\nu_d^{n-1}(I_j)\in\mod A] \\ [4mm]
= & \lim\limits_{n\to\infty} \frac{1}{n} \log \sum\limits_{i,j=1}^r |\chi_\mathbb{R}([P_i],[\nu_d^{n-1}(I_j)])| & [\nu_d^{n-1}(I_j)\in\mod A] \\ [4mm]
= & \lim\limits_{n\to\infty} \frac{1}{n} \log \sum\limits_{i,j=1}^r |\chi_\mathbb{R}([P_i],[\nu_d]^n([P_j]))| & ([\nu_d]([P_j])=(-1)^d[I_j]) \\ [4mm]
= & \lim\limits_{n\to\infty}\frac{1}{n}\log\|[\nu_d]^n\| & \\ [4mm]
= & \log\rho([\nu_d]) & \mbox{(Gelfand formula)} \\ [4mm]
= & \log\rho([\nu]) = \log\rho([S]) = \log\rho(\Phi). & ([\nu_d]=(-1)^d[\nu])
\end{array}$$
By Lemma~\ref{Lemma-Entropy-Shift}, we have $h_t(\nu)-dt = h_t(\nu_d) = \log\rho(\Phi)$.
So $h_t(\nu) = dt+\log\rho(\Phi)$.

From Lemma~\ref{Lemma-Entropy-Inverse}, we obtain $h(\nu) = h(\nu^{-1})$. Hence $\rho(\Phi) = \rho(\Phi^{-1})$.

\medskip

(2) Since $A$ is $d$-representation-infinite, we have $\nu_d^{-n}(A)\in\mod A$ for all $n\in\mathbb{Z}_{\ge 0}$.
Applying Lemma~\ref{Lemma-PolyEntropy-CohFun}, (1) and Lemma~\ref{Lemma-PolyEntropy-LinearOperator}, we obtain
$$\begin{aligned}
  & h^\mathrm{pol}_t(\nu^{-1}_d) \\
\stackrel{L4}{=} & \limsup\limits_{n\to\infty}
\frac{\log \sum\limits_{l\in\mathbb{Z}} \dim_k\Hom_{\D^b(A)}(A,\nu_d^{-n}(A)[l])\cdot e^{-lt}-n\cdot h_t(\nu_d^{-1})}{\log n} \\
= & \limsup\limits_{n\to\infty}
\frac{\log \sum\limits_{i,j=1}^r \dim_k \Hom_A(P_i,\nu_d^{-n}(P_j))-n\cdot h_t(\nu_d^{-1})}{\log n} \\
= & \limsup\limits_{n\to\infty}
\frac{\log \sum\limits_{i,j=1}^r |\chi([P_i],[\nu_d^{-n}(P_j)])|-n\cdot h_t(\nu_d^{-1})}{\log n} \\
\stackrel{(1)}{=} & \limsup\limits_{n\to\infty}
\frac{\log \|[\nu_d^{-1}]^n\|-n\cdot\log\rho([\nu_d^{-1}])}{\log n} \\
\stackrel{L7}{=} & s([\nu_d^{-1}]) = s([\nu^{-1}])= s([S^{-1}]) = s(\Phi^{-1}).
\end{aligned}$$

Similarly,
applying Lemma~\ref{Lemma-PolyEntropy-CohFun}, (1) and Lemma~\ref{Lemma-PolyEntropy-LinearOperator}, we obtain
$$\begin{aligned}
  & h^\mathrm{pol}_t(\nu_d) \\
\stackrel{L4}{=} & \limsup\limits_{n\to\infty}
\frac{\log \sum\limits_{l\in\mathbb{Z}} \dim_k\Hom_{\D^b(A)}(A,\nu_d^n(A)[l])\cdot e^{-lt}-n\cdot h_t(\nu_d)}{\log n} \\
= & \limsup\limits_{n\to\infty}
\frac{\log \sum\limits_{l\in\mathbb{Z}} \dim_k\Hom_{\D^b(A)}(A,\nu_d^{n-1}(A^*)[l-d])\cdot e^{-lt}-n\cdot h_t(\nu_d)}{\log n} \\
= & \limsup\limits_{n\to\infty}
\frac{\log \sum\limits_{i,j=1}^r \dim_k \Hom_{\D^b(A)}(P_i,\nu_d^{n-1}(I_j))-n\cdot h_t(\nu_d)}{\log n} \\
= & \limsup\limits_{n\to\infty}
\frac{\log \sum\limits_{i,j=1}^r |\chi([P_i],[\nu_d]^n([P_j]))|-n\cdot h_t(\nu_d)}{\log n} \\
\stackrel{(1)}{=} & \limsup\limits_{n\to\infty}
\frac{\log \|[\nu_d]^n\|-n\cdot\log\rho([\nu_d])}{\log n} \\
\stackrel{L7}{=} & s([\nu_d]) = s([\nu]) = s([S]) = s(\Phi).
\end{aligned}$$

By Lemma~\ref{Lemma-PolyEntropy-Shift} and Lemma~\ref{Lemma-PolyEntropy-Inverse}, we have $h^\mathrm{pol}_t(\nu) = h^\mathrm{pol}_t(\nu_d) = h^\mathrm{pol}_{-t}(\nu_d^{-1})= h^\mathrm{pol}_{-t}(\nu^{-1})$. Furthermore, $s(\Phi) = s(\Phi^{-1})$.

\medskip

(3) By \cite[Lemma 5]{Han20}, we have $C_{A^*}=C^T_A$.
Thanks to the Yomdin type inequality on Hochschild homology entropy in Theorem~\ref{Theorem-HHEntropy-Yomdin}, we obtain
$h^{HH_\bullet}(\tilde{S}) = h^{HH_\bullet}(A^*) \ge \log\rho(\Psi_{A^*}) = \log\rho(C_{A^*}C_A^{-1}) = \log\rho(C_A^TC_A^{-1}) = \log\rho(\Phi)$.
From (1) above, we know $h(S)=\log\rho(\Phi)$. Due to Theorem~\ref{Theorem-HHEntropy-Entropy}, we have $h^{HH_\bullet}(\tilde{S})\le h(S)$.
By Proposition~\ref{Proposition-HHCEntropy=HHEntropy}, we get $h^{HH^\bullet}(\tilde{S})=h^{HH_\bullet}(\tilde{S})$.
Thus $h^{HH^\bullet}(\tilde{S})=h^{HH_\bullet}(\tilde{S})=h(S)=\log\rho(\Phi)$.

\medskip

Similarly, by \cite[Theorem 1 (1) and Lemma 5]{Han20}, we have $C_{\RHom_A(A^*,A)} = (C^T_A)^TC^{-T}_AC_A = C_AC^{-T}_AC_A$. By the Yomdin type inequality on Hochschild homology entropy in Theorem~\ref{Theorem-HHEntropy-Yomdin}, we obtain $h^{HH_\bullet}(\tilde{S}^{-1}) = h^{HH_\bullet}(\RHom_A(A^*,A)) \ge \log\rho(\Psi_{\RHom_A(A^*,A)}) = \log\rho(C_{\RHom_A(A^*,A)}C_A^{-1}) = \log\rho(C_AC_A^{-T}) = \log\rho(\Phi^{-1})$. From (1) above, we know $h(S^{-1})=\log\rho(\Phi^{-1}) =\log\rho(\Phi)$.
It follows from Theorem~\ref{Theorem-HHEntropy-Entropy} that $h^{HH^\bullet}(\tilde{S}^{-1})\le h(S^{-1})$.
By Proposition~\ref{Proposition-HHCEntropy=HHEntropy}, we get $h^{HH^\bullet}(\tilde{S}^{-1})=h^{HH_\bullet}(\tilde{S}^{-1})$.
Thus $h^{HH^\bullet}(\tilde{S}^{-1})=h^{HH_\bullet}(\tilde{S}^{-1})=h(S^{-1})=\log\rho(\Phi^{-1})=\log\rho(\Phi)$.

\medskip

(4) For any $n\in\mathbb{Z}_{>0}$, from \cite[Proposition 2.9 (d)]{HerIyaOpp14}, we obtain
$$S^n(A) = \nu_d^{n-1}(A^*)[d(n-1)] \in (\mod A)[d(n-1)]$$
which is a stalk complex concentrating on degree $-d(n-1)$.
Thus we have
$$\ol{\mathrm{Sdim}}A = \ul{\mathrm{Sdim}}A = \lim\limits_{n\to \infty} \frac{-\inf S^n(A)}{n} = \lim\limits_{n\to \infty} \frac{-\sup S^n(A)}{n} = \lim\limits_{n\to \infty} \frac{d(n-1)}{n} = d.$$

Now we have finished the proof of the theorem. \end{proof}

\medskip

\begin{remark}{\rm
(1) Theorem~\ref{Theorem-RepInfAlg-Entropies} (1) generalizes \cite[Theorem 2.17]{DimHaiKatKon14}, and implies that the Gromov-Yomdin type equality on entropy, that is, $h(S)= \log\rho([S])$, holds for the Serre functor on perfect derived category of an elementary higher representation-infinite algebra.

(2) Theorem~\ref{Theorem-RepInfAlg-Entropies} (2) (cf. \cite[Proposition 4.4]{FanFuOuc21}) implies that the polynomial entropy of Serre functor on perfect derived category of an elementary higher representation-infinite algebra has nothing to do with the parameter $t$, that is, $h^\mathrm{pol}_t(S)$ is a constant function.

(3) Theorem~\ref{Theorem-RepInfAlg-Entropies} (3) implies that the Kikuta-Ouchi's question (Question~\ref{Question-Kikuta-Ouchi}) has positive answer, that is, $h^{HH^\bullet}(\tilde{S}^{\pm 1})=h^{HH_\bullet}(\tilde{S}^{\pm 1})=h(S^{\pm 1})$ and the Gromov-Yomdin type equality on Hochschild (co)homology entropy $h^{HH^\bullet}(\tilde{S}^{\pm 1})= \linebreak \log\rho([H^0(\tilde{S}^{\pm 1})])=h^{HH_\bullet}(\tilde{S}^{\pm 1})$ holds, for the (inverse) Serre quasi-functor on perfect dg module category of an elementary higher representation-infinite algebra.
}\end{remark}

\bigskip

\noindent {\footnotesize {\bf ACKNOWLEDGEMENT.} The author is very grateful to the anonymous referee for his/her careful reading the original manuscript and many conducive suggestions which make the paper much readable.}

\end{document}